\documentclass[a4paper]{article}
\usepackage[english]{babel}
\usepackage[cp1251]{inputenc}
\pdfoutput=1
\usepackage[colorlinks,linkcolor=blue,citecolor=blue,urlcolor=blue]{hyperref}
\usepackage{latexsym,amsfonts,amssymb,amsthm,amsmath,amsbsy,graphicx}
\newtheorem{thm}{Theorem}[section]

\newtheorem{cor}{Corollary}[section]

\theoremstyle{definition}

\begin{document}
\thispagestyle{empty}
\bigskip
 {\noindent \Large\bf\sc distribution of some functionals for a l\'{e}vy process with matrix-exponential
jumps of the same sign }\footnotemark[1]\footnotetext{This is an electronic reprint
  of the original article published in Theory of Stochastic Processes,
  Vol. 19(35), No. 1 (2014), 26-36. This reprint differs from the original
in pagination and typographic detail.}
\begin{center}
Ie.~Karnaukh
\end{center}
 \footnote{Department of Statistics and Probability Theory,
Dnipropetrovsk Oles Honchar National University,\\ 72, Gagarina pr.,
Dnipropetrovsk 49010, Ukraine.
 \href{mailto:ievgen.karnaukh@gmail.com}{ievgen.karnaukh@gmail.com}
 }
 \footnotemark[2]\footnotetext{The author would like to thank Prof. Dmytro Husak
for the useful comments, remarks and proofreading several draft versions of the paper.}
\begin{center}
\begin{quotation}
\noindent  {\small This paper provides a framework for investigations
in fluctuation theory for L\'evy processes with matrix-exponential jumps.
We present a matrix form of the components of the infinitely divisible factorization.
Using this representation we establish generalizations of some results known
for compound Poisson processes with exponential jumps in one
 direction and generally distributed jumps in the
other direction.}
\end{quotation}
\end{center}
\bigskip

L\'evy processes have many applications in practice as a base model in
risk theory, queuing and financial mathematics. Many problems can be connected to the fluctuation theory,
in which the factorization method plays a crucial role (see, for
instance, \cite{Kyprianou2006}).

The most studied class of L\'evy processes is the class of semi-continuous
processes (with L\'evy measure supported on a half-axis).
One of the factorization components
for the semi-continuous processes is entirely defined by the real root of the
cumulant equation
(or more specifically, by the right-inverse of the cumulant function).

This result can be generalized for L\'evy processes with
 matrix-exponential upward (or downward) jumps
 (see \cite{Bratiychuk1990, Lewis2008}
and references therein). For such processes one of the factorization components is
a rational function with finite number of poles, which are the (possibly complex)
 roots of the
cumulant equation. Using the properties of matrix-exponential distribution we can invert the
component to find the distribution of corresponding killed extremum.
 The convolution of the distribution
of this extremum and the integral transform of the L\'evy measure defines
the moment generating
function of other extremum.

We use the relations for the factorization components
to obtain in closed form the moment generating function
 of occupation time of a half-line (for semi-continuous case we refer to~\cite{Landriault2011}
  and for some other cases to \cite{Husak2011engl}).
Further generalization could be done for meromorphic L\'evy processes with the main difference
that the factorization components have infinitely many poles (see~\cite{Kuznetsov2012}).

\section{Matrix-exponential distribution}

The class of matrix-exponential (ME) distributions is the generalization
of exponential distribution and it comprises the phase-type distributions.
The ME class can be defined as a class of distributions
with a rational moment generating function (see \cite{Asmussen2010}).
The properties of ME distributions allow us to find in the closed form some generalizations
of the results known for L\'evy processes with exponential jumps.

A nonnegative random variable has a ME$\left(d\right)$ distribution
$\left(d\geq1\right)$, if its cumulative distribution function is
as follows
\[
F\left(x\right)=\begin{cases}
1+\boldsymbol{\beta}e^{\mathbf{R}x}\mathbf{R}^{-1}\mathbf{t} & x>0;\\
0 & x\leq0,
\end{cases}
\]
 where $\boldsymbol{\beta}$ is a $1\times d$ vector,
$\mathbf{R}$ is a non singular $d\times d$ matrix, $\mathbf{t}$
is a $d\times 1$ vector, and each possibly have complex entries.
The triple $\left(\boldsymbol{\beta},{\mathbf{R}},{\mathbf{t}}\right)$
is called a representation of the ME distribution. Note that, the same
distribution may have several representations. If the cumulative
function of a ME distribution can be represented as $F\left(x\right)=
1-\boldsymbol{\alpha}e^{\mathbf{T}x}\mathbf{e}, x>0$,
 where $\boldsymbol{\alpha}$
is a probability
vector and $\mathbf{T}$ is the intensity matrix of a Markov chain,
 $\mathbf{e}=\left(0,\ldots,0,1\right)^{\top}$, then the ME
distribution is called phase-type distribution. For details and general
results on matrix-exponential distributions we refer to \cite{Fackrell2003}.

For $x>0$ a ME distribution has density $f\left(x\right)=
\boldsymbol{\beta}e^{\mathbf{R}x}\mathbf{t}$,
which can be rewritten as (see \cite{Bratiychuk1990} and \cite{Lewis2008}):
\[
f\left(x\right)=\sum_{i=1}^{m}P_{i}\left(x\right)e^{-r_{i}x},
\]
 where $P_{i}\left(x\right)$ are polynomials of degree $k_{i}$,
$\Re[r_{m}]\geq\ldots\geq\Re[r_{2}]>r_{1}>0$ and $\sum_{i=1}^{m}k_{i}+m=d$.

If $p_{0}=1+\boldsymbol{\beta}\mathbf{R}^{-1}\mathbf{t}\neq0$, then the
ME distribution has nonzero mass at zero, and the moment generating function has the form
\[
\int_{0}^{\infty}e^{rx}dF\left(x\right)=p_{0}
-\boldsymbol{\beta}\left(r\mathbf{I}+\mathbf{R}\right)^{-1}\mathbf{t},
\quad\Re[r]=0.
\]
To find a representation of the distribution with known moment generating function
we can follow the approach given in \cite{Asmussen2010}.
Denote the vectors $\boldsymbol{\rho}=\left(\rho_{d},\ldots,\rho_{1}\right),$
$\left(\boldsymbol{\rho},1\right)=\left(\rho_{d},\ldots,\rho_{1},1\right)$,
$\mathbf{h}_{d}\left(r\right)=\left(1,r,\ldots,r^{d}\right)^{\top}$.
If the Laplace transform of $f\left(x\right)$ has the form
\begin{equation}\label{eq:1.1}
\int_{0}^{\infty}e^{-rx}f\left(x\right)dx=\frac{\beta_{1}r^{d-1}
+\beta_{2}r^{d-2}+\ldots+\beta_{d}}{r^{d}+\rho_{1}r^{d-1}+\ldots
+\rho_{d-1}r+\rho_{d}}=\frac{\boldsymbol{\beta}\mathbf{h}_{d-1}
\left(r\right)}{\left(\boldsymbol{\rho},1\right)\mathbf{h}_{d}\left(r\right)},
\end{equation}
then the corresponding density is $f\left(x\right)=\boldsymbol{\beta}e^{\mathbf{R}x}\mathbf{e}$,
$x>0$, where $\boldsymbol{\beta}=\left(\beta_{d},\ldots,\beta_{1}\right),$
$\mathbf{e}=\left(0,\ldots,0,1\right)^{\top},$ and
$\mathbf{R}=\begin{pmatrix}0 & 1 & \ldots & 0\\
\vdots & \vdots & \ddots & \vdots\\
0 & 0 & \ldots & 1\\
-\rho_{d} & -\rho_{d-1} & \ldots & -\rho_{1}
\end{pmatrix}=\begin{pmatrix}0\quad\mathbf{I}\\
-\boldsymbol{\rho}
\end{pmatrix}$.

In the case when a ME distribution is defined on negative
half-axis, we can
follow the similar reasoning. If
\begin{equation}\label{eq:1.2}
\int_{-\infty}^{0}e^{rx}f\left(x\right)dx=
\frac{\boldsymbol{\beta}\mathbf{h}_{d-1}\left(r\right)}
{\left(\boldsymbol{\rho},1\right)\mathbf{h}_{d}\left(r\right)},
\end{equation}
then  $f\left(x\right)=\boldsymbol{\beta}e^{\mathbf{R}x}\mathbf{e}$,
$x<0$, where $\boldsymbol{\beta}=\left(\beta_{d},\ldots,\beta_{1}\right),$
$\mathbf{R}=\begin{pmatrix}0\quad-\mathbf{I}\\
\boldsymbol{\rho}
\end{pmatrix}$. If a ME distribution has support on the entire real line,
then it is called bilateral matrix-exponential distribution (see~\cite{Bladt2012}).

\section{Extrema and overshoot}

Let us suppose that $X_{t},t\geq0$ is a L\'evy process with cumulant
function
\[
k\left(r\right)=a'r+\frac{\sigma^{2}}{2}r^{2}+
\int_{-\infty}^{\infty}\left(e^{rx}-1-rxI_{\left\{ |x|\leq1\right\} }\right)
\Pi\left(dx\right),
\]
where $a'$ is a real constant, $\sigma>0$, and $\Pi$ is
a non negative
measure, defined on $R\backslash\{0\}$:
$\int_{R}\min\left\{ x^{2},1\right\} \Pi\left(dx\right)<\infty.$

Throughout we impose the restriction that
$\int_{-1}^{1}|x|\Pi\left(dx\right)<\infty,$
then the cumulant function can be represented as
 follows
\begin{equation}\label{eq:2.1}
k\left(r\right)=ar+\frac{\sigma^{2}}{2}r^{2}+
\int_{-\infty}^{\infty}
\left(e^{rx}-1\right)\Pi\left(dx\right),
\end{equation}
where $a=a'-\int_{-1}^{1}|x|\Pi\left(dx\right).$

Denote by $\theta_{s}$ an exponential random variable with parameter
$s>0$: $\mathsf{P}\left\{ \theta_{s}>t\right\} =e^{-st},t>0,$ independent
of process $X_{t},$ and by definition $\theta_{0}=\infty.$ Then $X_{\theta_{s}}$
is called a L\'evy process killed at rate $s$ (see \cite{Bertoin1996}).
For the moment generating function of $X_{\theta_{s}}$
\[
\mathsf{E}e^{rX_{\theta_{s}}}=\frac{s}{s-k\left(r\right)}
\]
the identity of infinitely divisible factorization takes place
\[
\mathsf{E}e^{rX_{\theta_{s}}}=\mathsf{E}e^{rX_{\theta_{s}}^{+}}\mathsf{E}e^{rX_{\theta_{s}}^{-}},\Re\left[r\right]=0,
\]
where $X_{\theta_{s}}^{+}=\sup_{0\leq t\leq\theta_{s}}X_{t}$, $X_{\theta_{s}}^{-}=\inf_{0\leq t\leq\theta_{s}}X_{t}$
are the supremum and infimum of the process, respectively, killed
at rate $s$.

In general case the closed formulae for factorization components are not known, so
we should impose additional restrictions on the parameters of the process.
Following~\cite{Lewis2008}, we consider L\'evy processes that have finite intensity negative (or positive)
jumps with ME distribution, arbitrary
positive (negative) jumps, and possibly drift and gaussian component.

That is, we assume that
$\Pi\left(dx\right)=\lambda_{-}\sum_{i=1}^{m_{-}}P_{i}^{-}\left(x\right)e^{b_{i}x}dx$, $x<0$ (or correspondingly $\Pi\left(dx\right)=\lambda_{+}\sum_{i=1}^{m_{+}}P_{i}^{+}\left(x\right)e^{-c_{i}x}dx$,
$x>0$)
where $\lambda_{\pm}=\int_{R^{\pm}}\Pi\left(dx\right)<\infty$,
$P_{i}^{\pm}\left(x\right)$
are the polynomials of degree $k_{i}^{\pm}$,
$\sum_{i=1}^{m_{\pm}}k_{i}^{\pm}+m_{\pm}=d_{\pm}$,
$\Re[b_{m_{-}}]\geq\ldots\geq\Re[b_{2}]>b_{1}>0$ and
 $\Re[c_{m_{+}}]\geq\ldots\geq\Re[c_{2}]>c_{1}>0$.
Also we split two cases (for details, see~\cite{Lewis2008}):
\begin{description}
  \item[$\left(NS\right)_{\pm}$] $\sigma>0$ or $\sigma=0,\pm a>0$,
  \item[$\left(S\right)_{\pm}$] $\sigma=0,\mp a\geq0$,
\end{description}
where sign '+'  corresponds to the case when positive jumps have ME distribution
while sign '--' corresponds to the case when negative jumps have ME distribution.

Due to \cite{Bratiychuk1990,Lewis2008}, in any of the cases
$\left(NS\right)_{\pm}$ or $\left(S\right)_{\pm}$ the moment generating
function
$\mathsf{E}e^{rX_{\theta_{s}}^{\pm}}$ is a rational function
and $X_{\theta_{s}}^{\pm}$ has a ME distribution. In addition, the
cumulant equation
\[
k\left(r\right)=s
\]
has the roots $\left\{ \pm r_{i}^{\pm}
\left(s\right)\right\} _{i=1}^{N_{\pm}}$
in half-plane $\pm\Re[r]>0$,
 where $N_{\pm}=\begin{cases}
d_{\pm}+1 & \left(NS\right)_{\pm},\\
d_{\pm} & \left(S\right)_{\pm},
\end{cases}$ $r_{1}^{+}\left(s\right)$
is the unique root on $\left[0,c_{1}\right]$
($-r_{1}^{-}\left(s\right)$ is the unique root
 on $\left[-b_{1},0\right]$).
These roots entirely define the distribution of corresponding extrema.

Write
\begin{gather*}
\beta_{k}^{-}=\sum_{1\leq i_{1}<\ldots<i_{k}\leq N_{-}}b_{i_{1}}
\ldots b_{i_{k}},\beta_{k}^{+}=\sum_{1\leq i_{1}<\ldots<i_{k}\leq
N_{+}}c_{i_{1}}\ldots c_{i_{k}},\\\rho_{k}^{\pm}\left(s\right)=
\sum_{1\leq i_{1}<\ldots<i_{k}\leq N_{\pm}}r_{i_{1}}^{\pm}\left(s\right)
\ldots r_{i_{k}}^{\pm}\left(s\right),
\end{gather*}
then the distribution of $X_{\theta_{s}}^{\pm}$ can be represented by
the parameters:
\[
\boldsymbol{\beta}_{\pm}=\left(\beta_{d_{\pm}}^{\pm},
\ldots,\beta_{1}^{\pm}\right),\boldsymbol{\rho}_{\pm}\left(s\right)=
\left(\rho_{N_{\pm}}^{\pm}\left(s\right),\ldots,
\rho_{1}^{\pm}\left(s\right)\right),\mathbf{\mathbf{R}}_{\pm}\left(s\right)
=\begin{pmatrix}0\quad\mathbf{\pm I}\\
\mp\boldsymbol{\rho}_{\pm}\left(s\right)
\end{pmatrix}.
\]
 Under additional conditions the moment generating function of $X_{\theta_{s}}^{\mp}$ we
can determine in terms of integral transforms of the L\'evy measure:
\[
\overline{\Pi}^{+}\left(x\right)=\int_{x}^{\infty}\Pi\left(dx\right),x>0;
\overline{\Pi}^{-}\left(x\right)=\int_{-\infty}^{x}\Pi\left(dx\right),x<0;
\tilde{\Pi}^{\pm}\left(r\right)=\int_{R^{\pm}}e^{rx}
\overline{\Pi}^{\pm}\left(x\right)dx.
\]
The following statement is essentially based on the results given
in \cite{Bratiychuk1990}
and \cite{Lewis2008}.
\begin{thm}\label{thm:2.1}
If for L\'evy process $X_t$ one of the cases $\left(NS\right)_{-}$ or
$\left(S\right)_{-}$ holds, then
\begin{equation}\label{eq:2.2}
P'_{-}\left(s,x\right)=\frac{\partial}{\partial x}\mathsf{P}
\left\{ X_{\theta_{s}}^{-}<x\right\} =\mathbf{q}_{-}\left(s\right)
e^{\mathbf{R}_{-}\left(s\right)x}\mathbf{e},x<0,
\end{equation}
where
\begin{equation}\label{eq:2.3}
\mathbf{q}_{-}\left(s\right)=\begin{cases}
\frac{\rho_{d_{-}+1}^{-}\left(s\right)}{\beta_{d_{-}}^{-}}
\left(\boldsymbol{\beta}_{-},1\right) & \left(NS\right)_{-},\\
\frac{\rho_{d_{-}}^{-}\left(s\right)}{\beta_{d_{-}}^{-}}
\left(\boldsymbol{\beta}_{-}-\boldsymbol{\rho}_{-}\left(s\right)\right)
& \left(S\right)_{-}.
\end{cases}
\end{equation}
Moreover, in case $\left(S\right)_{-}$:
 $p_{-}\left(s\right)=\mathsf{P}\left\{ X_{\theta_{s}}^{-}=0\right\}
  \neq0$
and $p_{-}\left(s\right)=\frac{\rho_{d_{-}}^{-}\left(s\right)}
{\beta_{d_{-}}^{-}}=\left(\prod_{i=1}^{d_{-}}\frac{r_{i}^{-}
\left(s\right)}{b_{i}}\right)$.

If additionally $\mathsf{D}X_{1}<\infty,
\mathsf{E}X_{\theta_{s}}^{+}<\infty$,
the moment generating function of $X_{\theta_{s}}^{+}$
could be represented
as
\begin{equation}\label{eq:2.4}
\mathsf{E}e^{rX_{\theta_{s}}^{+}}=\left(1-\frac{r}{s}
\left(A_{*}^{-}\left(s\right)+\mathsf{E}
e^{rX_{\theta_{s}}^{-}}\tilde{\Pi}^{+}\left(r\right)
-\mathbf{q}_{-}\left(s\right)\left(r\mathbf{I}
+\mathbf{R}_{-}\left(s\right)\right)^{-1}\tilde{\Pi}^{+}
\left(-\mathbf{R}_{-}\left(s\right)\right)\mathbf{e}\right)\right)^{-1},
\end{equation}
\[
A_{*}^{-}\left(s\right)=\left\{ \begin{array}{lc}
\frac{\sigma^{2}}{2}\left.\frac{\partial}{\partial y}
\mathsf{P}\left\{ X_{\theta_{s}}^{-}<y\right\} \right|_{y=0}
 & \sigma>0,\\
\mathsf{P}\left\{ X_{\theta_{s}}^{-}=0\right\}
\max\left\{ 0,a\right\}  & \sigma=0,
\end{array}\right.
=\begin{cases}
\frac{\sigma^2}{2}\frac{\rho_{d_{-}+1}^{-}\left(s\right)}{\beta_{d_{-}}^{-}}& \left(NS\right)_{-},\\
a \frac{\rho_{d_{-}}^{-}\left(s\right)}{\beta_{d_{-}}^{-}}
& \left(S\right)_{-}.
\end{cases}
\]
Denote the first passage time by $\tau_{x}^{+}=
\inf\left\{ t>0:X_{t}>x\right\} $,
then the distribution of discounted overshoot
$X_{\tau_{x}^{+}}-x$, $x>0$,
is defined by
\begin{multline}\label{eq:2.5}
\mathsf{E}\left[e^{-s\tau_{x}^{+}},X_{\tau_{x}^{+}}-x\in dv,
\tau_{x}^{+}<\infty\right]=s^{-1}A_{*}^{-}\left(s\right)
\frac{\partial}{\partial x}\mathsf{P}\left\{ X_{\theta_{s}}^{+}<x\right\}
\delta\left(v\right)dv+\\
s^{-1}\int_{0}^{x}\left(p_{-}\left(s\right)\Pi\left(dv+y\right)
+\int_{y}^{\infty}\Pi\left(dv+z\right)\mathbf{q}_{-}\left(s\right)
e^{\mathbf{R}_{-}\left(s\right)\left(y-z\right)}\mathbf{e}dz\right)
\mathsf{P}\left\{ X_{\theta_{s}}^{+}\in x-dy\right\},
\end{multline}
where $\delta(v)$ is the Dirac delta function.
\end{thm}
\begin{proof} If one of the cases $\left(NS\right)_{-}$ or
$\left(S\right)_{-}$ holds, then according to \cite{Lewis2008} the moment
generating function of
$X_{\theta_{s}}^{-}$ can be defined by the relation
\begin{equation}\label{eq:2.6}
\mathsf{E}e^{rX_{\theta_{s}}^{-}}=
\frac{\prod_{i=1}^{N_{-}}r_{i}^{-}\left(s\right)}
{\prod_{i=1}^{d_{-}}b_{i}}\frac{\prod_{i=1}^{d_{-}}\left(r+b_{i}\right)}
{\prod_{i=1}^{N_{-}}\left(r+r_{i}^{-}\left(s\right)\right)}.
\end{equation}
Using notation given above we can
rewrite this relation as
\begin{equation}\label{eq:2.7}
\mathsf{E}e^{rX_{\theta_{s}}^{-}}=
\frac{\rho_{N_{-}}^{-}\left(s\right)}
{\beta_{d_{-}}^{-}}\frac{\left(\boldsymbol{\beta}_{-},1\right)
\mathbf{h}_{d_{-}}\left(r\right)}{\left(\boldsymbol{\rho}_{-}
\left(s\right),1\right)\mathbf{h}_{N_{-}}\left(r\right)}=
\begin{cases}
\frac{\rho_{d_{-}+1}^{-}\left(s\right)}{\beta_{d_{-}}^{-}}
\frac{\left(\boldsymbol{\beta}_{-},1\right)\mathbf{h}_{d_{-}}
\left(r\right)}{\left(\boldsymbol{\rho}_{-}\left(s\right),1\right)
\mathbf{h}_{d_{-}+1}\left(r\right)} & \left(NS\right)_{-},\\
\frac{\rho_{d_{-}}^{-}\left(s\right)}{\beta_{d_{-}}^{-}}
\left(1+\frac{\left(\boldsymbol{\beta}_{-}-\boldsymbol{\rho}_{-}
\left(s\right)\right)\mathbf{h}_{d_{-}-1}\left(r\right)}
{\left(\boldsymbol{\rho}_{-}\left(s\right),1\right)
\mathbf{h}_{d_{-}}\left(r\right)}\right) & \left(S\right)_{-}.
\end{cases}
\end{equation}
which allows for inversion in $r$, so we get (\ref{eq:2.2}) and
(\ref{eq:2.3}).

Under conditions of the theorem (see \cite[Corollary 2.2]{Husak2011engl}):
\begin{equation}\label{eq:2.8}
\mathsf{E}e^{rX_{\theta_{s}}^{+}}=
\left(1-s^{-1}r\left(A_{*}^{-}\left(s\right)
+\int_{0}^{\infty}e^{rx}\int_{-\infty}^{0}\overline{\Pi}^{+}
\left(x-y\right)dP_{-}\left(s,y\right)\right)\right)^{-1}.
\end{equation}
Using (\ref{eq:2.2}) and (\ref{eq:2.3}) we get
\begin{multline*}
\int_{0}^{\infty}e^{rx}\int_{-\infty}^{0}
\overline{\Pi}^{+}\left(x-y\right)dP_{-}\left(s,y\right)=\\
=p_{-}\left(s\right)\tilde{\Pi}^{+}\left(r\right)+
\mathbf{q}_{-}\left(s\right)\left(r\mathbf{I}+
\mathbf{R}_{-}\left(s\right)\right)^{-1}\left(\tilde{\Pi}^{+}\left(r\right)
-\tilde{\Pi}^{+}\left(-\mathbf{R}_{-}\left(s\right)\right)\right)\mathbf{e}_{-}=\\
=\mathsf{E}e^{rX_{\theta_{s}}^{-}}\tilde{\Pi}^{+}
\left(r\right)-\mathbf{q}_{-}\left(s\right)\left(r\mathbf{I}+
\mathbf{R}_{-}\left(s\right)\right)^{-1}\tilde{\Pi}^{+}
\left(-\mathbf{R}_{-}\left(s\right)\right)\mathbf{e}_{-},
\end{multline*}
Substituting the last relation in (\ref{eq:2.8}) yields (\ref{eq:2.4}).

Using formula (\ref{eq:2.2}), relation (\ref{eq:2.5}) can be deduce
by integration of the Gerber-Shiu measure (see
\cite{Kuznetsov2012}):
\begin{multline*}
\mathsf{E}\left[e^{-s\tau_{x}^{+}},x-X_{\tau_{x}^{+}-0}^{+}
\in dy,x-X_{\tau_{x}^{+}-0}\in dz,X_{\tau_{x}^{+}}-x\in dv,
\tau_{x}^{+}<\infty\right]=\\
=s^{-1}\mathsf{P}\left\{ X_{\theta_{s}}^{+}\in x-dy\right\}
 \mathsf{P}\left\{ -X_{\theta_{s}}^{-}\in dz-v\right\}
 \Pi\left(dv+z\right),\; v,z>0,0\leq y\leq\min\left\{ x,z\right\} ,
\end{multline*}
with respect to $y$ and $z$ and taking into account that $\mathsf{E}\left[e^{-s\tau_{x}^{+}},
x-X_{\tau_{x}^{+}-0}=0,
\tau_{x}^{+}<\infty\right]=s^{-1}A_{*}^{-}\left(s\right)
\frac{\partial}{\partial x}\mathsf{P}\left\{ X_{\theta_{s}}^{+}<x\right\}$
(see \cite[(2.55)]{Husak2011engl}).
\end{proof}
\begin{cor}\label{cor:2.1}
If one of the cases $\left(NS\right)_{+}$ or $\left(S\right)_{+}$ holds,
then
\begin{equation}\label{eq:2.9}
P'_{+}\left(s,x\right)=\frac{\partial}{\partial x}\mathsf{P}
\left\{ X_{\theta_{s}}^{+}<x\right\} =\mathbf{q}_{+}\left(s\right)
e^{\mathbf{R}_{+}\left(s\right)x}\mathbf{e},x>0,
\end{equation}
where
\begin{equation}\label{eq:2.10}
\mathbf{q}_{+}\left(s\right)=\begin{cases}
\frac{\rho_{d_{+}+1}^{+}\left(s\right)}
{\beta_{d_{+}}^{+}}\left(\boldsymbol{\beta}_{+},1\right)
 & \left(NS\right)_{+},\\
\frac{\rho_{d_{+}}^{+}\left(s\right)}{\beta_{d_{+}}^{+}}
\left(\boldsymbol{\beta}_{+}-\boldsymbol{\rho}_{+}\left(s\right)\right)
 & \left(S\right)_{+}.
\end{cases}
\end{equation}
 Moreover, in the case $\left(S\right)_{+}$:
 $p_{+}\left(s\right)=\mathsf{P}\left\{ X_{\theta_{s}}^{+}=0\right\}
  \neq0$
and $p_{+}\left(s\right)=\rho_{d_{+}}^{+}
\left(s\right)/\beta_{d_{+}}^{+}$.

If additionally, $\mathsf{D}X_{1}<\infty,\mathsf{E}X_{\theta_{s}}^{-}<\infty$,
then the moment generating function of $X_{\theta_{s}}^{-}$ admits
the representation
\begin{equation}\label{eq:2.11}
\mathsf{E}e^{rX_{\theta_{s}}^{-}}=\left(1+\frac{r}{s}
\left(A_{*}^{+}\left(s\right)+\mathsf{E}e^{rX_{\theta_{s}}^{+}}
\tilde{\Pi}^{-}\left(r\right)+\mathbf{q}_{-}\left(s\right)
\left(r\mathbf{I}+\mathbf{R}_{+}\left(s\right)\right)^{-1}\tilde{\Pi}^{-}
\left(-\mathbf{R}_{+}\left(s\right)\right)\mathbf{e}\right)\right)^{-1},
\end{equation}
 where
\[
A_{*}^{+}\left(s\right)=\begin{cases}
\frac{\sigma^2}{2}\frac{\rho_{d_{+}+1}^{+}\left(s\right)}{\beta_{d_{+}}^{+}}& \left(NS\right)_{+},\\
a \frac{\rho_{d_{+}}^{+}\left(s\right)}{\beta_{d_{+}}^{+}}
& \left(S\right)_{+}.
\end{cases}
\]
\end{cor}
\begin{proof} To prove relations (\ref{eq:2.9}) -- (\ref{eq:2.11}) we
can use (\ref{eq:2.2}) -- (\ref{eq:2.4}) and the fact that if for the
dual process $Y_{t}=-X_{t}$ one of the cases $\left(NS\right)_{-}$
or $\left(S\right)_{-}$ holds, then for $X_{t}$ we have the case $\left(NS\right)_{+}$
or $\left(S\right)_{+}$ correspondingly, and $\mathsf{E}e^{rY_{\theta_{s}}^{+}}
=\mathsf{E}e^{-rX_{\theta_{s}}^{-}}$.
\end{proof}

If we have cases $\left(NS\right)_{-}$ and $\left(NS\right)_{+}$ ($\left(S\right)_{-}$
 and
$\left(S\right)_{+}$) at the same time, then the L\'evy process $X_{t}$ has
the gaussian part with possibly drift (with
zero drift and without gaussian part, correspondingly) and the jump part is a compound Poisson process
with bilateral matrix-exponential distributed jumps.
\begin{cor}\label{cor:2.2}
If we have the cases $\left(NS\right)_{-}$
and $\left(NS\right)_{+}$ at the same time, then
\begin{equation}\label{eq:2.12}
\frac{\partial}{\partial x}\mathsf{P}
\left\{ X_{\theta_{s}}^{\pm}<x\right\} =\mathbf{q}_{\pm}\left(s\right)
e^{\mathbf{R}_{\pm}\left(s\right)x}\mathbf{e},\pm x>0,
\end{equation}
where $\mathbf{q}_{\pm}\left(s\right)=
\frac{\rho_{d_{\pm+1}}^{\pm}\left(s\right)}{\beta_{d\pm}^{\pm}}
\left(\boldsymbol{\beta}_{\pm},1\right)$.

If we have $\left(S\right)_{-}$ and $\left(S\right)_{+}$ simultaneously,
 then $\mathsf{P}\left\{ X_{\theta_{s}}^{\pm}=0\right\}
 =\rho_{d_{\pm}}^{\pm}\left(s\right)/\beta_{d_{\pm}}^{\pm}$
and
\begin{equation}\label{eq:2.13}
\frac{\partial}{\partial x}
\mathsf{P}\left\{ X_{\theta_{s}}^{\pm}<x\right\} =
\mathbf{q}_{\pm}\left(s\right)e^{\mathbf{R}_{\pm}\left(s\right)x}\mathbf{e},
\pm x>0,
\end{equation}
where $\mathbf{q}_{\pm}\left(s\right)=
\frac{\rho_{d_{\pm}}^{\pm}\left(s\right)}{\beta_{d_{\pm}}^{\pm}}
\left(\boldsymbol{\beta}_{\pm}-\boldsymbol{\rho}_{\pm}\left(s\right)\right)$.
\end{cor}
Note that, if $\sigma=0,a\geq0,$ and $\Pi\left(dx\right)=\lambda_{-}b_{1}e^{b_{1}x}dx$ for
$x<0$, then the process $X_{t}$ is called almost lower semi-continuous
(for details see \cite{Husak2011engl}) and
we have the case $\left(S\right)_{-}$ with $d_{-}=1$, hence $N_{-}=1$
and the cumulant equation has a unique negative real root
$-r_{1}^{-}\left(s\right)>-b_{1}$.
Hence, by (\ref{eq:2.2}) the density of infimum is $P'_{-}\left(s,x\right)
=\frac{r_{1}^{-}\left(s\right)}{b_{1}}
\left(b_{1}-r_{1}^{-}\left(s\right)\right)e^{r_{1}^{-}\left(s\right)x},x>0$
(cf. \cite[(3.110)]{Husak2011engl}).

To find the distribution of absolute supremum $X^{+}=\sup_{0\leq t<\infty}X_{t}$ or infimum $X^{-}=
\inf_{0\leq t<\infty}X_{t}$ we should take into consideration the sign of $\mathsf{E} X_{1}$.
If $\mu=\mathsf{E}X_{1}<0$ ($\mu>0$), then the distribution of $X^+$ ($X^-$)
is non degenerate and it is defined in terms of the roots of the cumulant
equation for $s=0$ (see, for instance, \cite{Husak2011engl}).

If we have one of the cases $\left(NS\right)_{\pm}$ or $\left(S\right)_{\pm}$,
then from \cite{Bratiychuk1990} it can be seen that for $\pm\mu<0$:
$r_{i}^{\pm}\left(s\right)\underset{s\rightarrow0}{\longrightarrow}r_{i}^{\pm},
\Re\left[r_{i}^{\pm}\right]>0$,
$i={1,\ldots,N_{\pm}}$, and for $\mp\mu<0$:
$r_{i}^{\pm}\left(s\right)\underset{s\rightarrow0}{\longrightarrow}r_{i}^{\pm},
\Re\left[r_{i}^{\pm}\right]>0$,
$i={2,\ldots,N_{\pm}}$, $r_{1}^{\pm}\left(s\right)
\underset{s\rightarrow0}{\longrightarrow}0$,
$s^{-1}r_{1}^{\pm}\left(s\right)\underset{s\rightarrow0}{\longrightarrow}|\mu|$.
Thus we obtain the next corollary of Theorem \ref{thm:2.1}.
\begin{cor}\label{cor:2.3}
Let one of the cases $\left(NS\right)_{-}$ or $\left(S\right)_{-}$ hold
and $\mu=\mathsf{E}X_{1}<0$, then
\begin{equation}\label{eq:2.14}
\lim_{s\rightarrow0}s^{-1}P'_{-}\left(s,x\right)
=\mathbf{q}'_{-}e^{\mathbf{R}_{-}\left(0\right)x}\mathbf{e},x<0,
\end{equation}
\begin{equation}\label{eq:2.15}
\mathbf{q}'_{-}=\begin{cases}
{\displaystyle \frac{\prod_{i=2}^{d_{-}+1}r_{i}^{-}}{|\mu|\beta_{d_{-}}^{-}}
\left(\boldsymbol{\beta}_{-},1\right)} & \left(NS\right)_{-},\\
{\displaystyle \frac{\prod_{i=2}^{d_{-}}r_{i}^{-}}{|\mu|
\beta_{d_{-}}^{-}}\left(\boldsymbol{\beta}_{-}-\boldsymbol{\rho}_{-}
\left(0\right)\right)} & \left(S\right)_{-}.
\end{cases}
\end{equation}
 Moreover, in case $\left(S\right)_{-}$:
  $p'_{-}=\lim_{s\rightarrow0}s^{-1}p_{-}\left(s\right)
 =\left(\prod_{i=2}^{d_{-}}r_{i}^{-}\right)/\left(|\mu|
 \prod_{i=1}^{d_{-}}b_{i}\right)$.

If additionally $\mathsf{D}X_{1}<\infty$,
then the moment generating function of $X^{+}$
has the form
\begin{multline}\label{eq:2.16}
\mathsf{E}e^{rX^{+}}=\\=\left(1-r\left(A'_{-}+p'_{-}\tilde{\Pi}^{+}\left(r\right)
+\mathbf{q}'_{-}\left(r\mathbf{I}+\mathbf{R}_{-}\left(0\right)
\right)^{-1}\left(\tilde{\Pi}^{+}\left(r\right)
\mathbf{I}-\tilde{\Pi}^{+}\left(-\mathbf{R}_{-}\left(0\right)
\right)\right)\mathbf{e}\right)\right)^{-1},
\end{multline}
 \begin{equation*}
A'_{-}=\lim_{s\rightarrow0}s^{-1}A_{*}^{-}\left(s\right)=\begin{cases}
{\frac{\sigma^2}{2}\frac{\prod_{i=2}^{d_{-}+1}r_{i}^{-}}{|\mu|\beta_{d_{-}}^{-}}}
 & \left(NS\right)_{-},\\
{ a\frac{\prod_{i=2}^{d_{-}}r_{i}^{-}}{|\mu|
\beta_{d_{-}}^{-}}} & \left(S\right)_{-}.
\end{cases}
\end{equation*}

The distribution of the overjump is defined by the relation
\begin{multline}\label{eq:2.17}
\mathsf{P}\left\{ X_{\tau_{x}^{+}}-x\in dv\right\} =
A'_{-}\frac{\partial}{\partial x}\mathsf{P}\left\{ X^{+}<x\right\}
 \delta\left(v\right)dv+\\
\int_{0}^{x}\left(p'_{-}\Pi\left(dv+y\right)+\int_{y}^{\infty}
\Pi\left(dv+z\right)\mathbf{q}'_{-}e^{\mathbf{R}_{-}\left(0\right)
\left(y-z\right)}\mathbf{e}dz\right)\mathsf{P}\left\{ X^{+}\in x-dy\right\}.
\end{multline}
\end{cor}

\section{Occupation time and ladder process}

Denote the moment generating function for the time that the process $X_{t}$
spends in the interval $\left(x,+\infty\right)$ until $\theta_{s}$ by
\[
D_{x}\left(s,u\right)=\mathsf{E}e^{-u\int_{0}^{\theta_{s}}I\left\{ X_{t}>x\right\} dt}.
\]
Combining the results of the previous section with the relations for $D_{x}\left(s,u\right)$
given in~\cite{Husak2011engl} yields the following statement.
\begin{thm}\label{thm:3.1}
If one of the cases $\left(NS\right)_{-}$
 or $\left(S\right)_{-}$
$\left(a>0\right)$ holds, then
\[
D_{0}(s,u)=\frac{s}{s+u}\frac{\rho_{N_{-}}^{-}\left(s+u\right)}
{\rho_{N_{-}}^{-}\left(s\right)},
\]
\begin{multline}\label{eq:3.1}
D_{x}\left(s,u\right)=
\frac{s}{s+u}\times\\\times\left(1-\frac{\rho_{N_{-}}^{-}\left(s+u\right)}
{\rho_{N_{-}}^{-}\left(s\right)}\left(\boldsymbol{\rho}_{-}\left(s\right)
-\boldsymbol{\rho}_{-}\left(s+u\right)\right)\mathbf{R}_{-}^{-1}
\left(s+u\right)e^{\mathbf{R}_{-}\left(s+u\right)x}\mathbf{e}\right),
x<0.
\end{multline}

If one of the cases $\left(NS\right)_{+}$ or $\left(S\right)_{+}$
$\left(a<0\right)$ holds, then
\[
D_{0}(s,u)=\frac{\rho_{N_{+}}^{+}\left(s\right)}{\rho_{N_{+}}^{+}\left(s+u\right)},
\]
\begin{equation}\label{eq:3.2}
D_{x}\left(s,u\right)=1+\frac{\rho_{N_{+}}^{+}\left(s\right)}
{\rho_{N_{+}}^{+}\left(s+u\right)}\left(\boldsymbol{\rho}_{+}
\left(s+u\right)-\boldsymbol{\rho}_{+}\left(s\right)\right)
\mathbf{R}_{+}^{-1}\left(s\right)e^{\mathbf{R}_{+}\left(s\right)x}
\mathbf{e},x>0.
\end{equation}
\end{thm}
\begin{proof} Following \cite[Theorem 2.6]{Husak2011engl},
for non step-wise processes the next relations are true
\begin{gather}
\int_{-0}^{+\infty}e^{rx}D'_{x}\left(s,u\right)dx+D_{0}(s,u)=
\frac{\mathsf{E}e^{rX_{\theta_{s}}^{+}}}
{\mathsf{E}e^{rX_{\theta_{s+u}}^{+}}},\nonumber \\
\int_{-\infty}^{+0}e^{rx}D'_{x}\left(s,u\right)dx-D_{0}(s,u)
=-\frac{s}{s+u}\frac{\mathsf{E}e^{rX_{\theta_{s+u}}^{-}}}
{\mathsf{E}e^{rX_{\theta_{s}}^{-}}}.\label{eq:3.3}
\end{gather}
If one of the cases $\left(NS\right)_{-}$ or
$\left(S\right)_{-}$
$\left(a>0\right)$ holds, then recalling Theorem \ref{thm:2.1}
it follows
that
\begin{multline*}
\frac{\mathsf{E}e^{rX_{\theta_{s+u}}^{-}}}
{\mathsf{E}e^{rX_{\theta_{s}}^{-}}}
=\frac{\rho_{N_{-}}^{-}\left(s+u\right)}{\rho_{N_{-}}^{-}
\left(s\right)}\frac{\left(\boldsymbol{\rho}_{-}\left(s\right),1\right)
\mathbf{h}_{N_{-}}\left(r\right)}{\left(\boldsymbol{\rho}_{-}
\left(s+u\right),1\right)\mathbf{h}_{N_{-}}\left(r\right)}=\\=
\frac{\rho_{N_{-}}^{-}\left(s+u\right)}{\rho_{N_{-}}^{-}\left(s\right)}
\left(1-\frac{\left(\boldsymbol{\rho}_{-}\left(s+u\right)-
\boldsymbol{\rho}_{-}\left(s\right)\right)\mathbf{h}_{N_{-}-1}
\left(r\right)}{\left(\boldsymbol{\rho}_{-}\left(s+u\right),
1\right)\mathbf{h}_{N_{-}}\left(r\right)}\right).
\end{multline*}
Taking account of formula (\ref{eq:3.3})
this gives us the following
\[
D'_{x}\left(s,u\right)=\frac{s}{s+u}
\frac{\rho_{N_{-}}^{-}\left(s+u\right)}{\rho_{N_{-}}^{-}\left(s\right)}
\left(\boldsymbol{\rho}_{-}\left(s+u\right)-\boldsymbol{\rho}_{-}\left(s\right)
\right)
e^{\mathbf{R}_{-}\left(s+u\right)x}\mathbf{e},\;
x<0,
\]
and combining with $\lim_{x\rightarrow-\infty}D_{x}\left(s,u\right)
=\frac{s}{s+u}$,
we receive (\ref{eq:3.1}).

Similarly, in case $\left(NS\right)_{+}$ or
 $\left(S\right)_{+}$
$\left(a<0\right)$:
\[
\frac{\mathsf{E}e^{-rX_{\theta_{s}}^{+}}}
{\mathsf{E}e^{-rX_{\theta_{s+u}}^{+}}}=
\frac{\rho_{N_{+}}^{+}\left(s\right)}{\rho_{N_{+}}^{+}
\left(s+u\right)}\left(1+\frac{\left(\boldsymbol{\rho}_{+}
\left(s+u\right)-\boldsymbol{\rho}_{+}\left(s\right)\right)
\mathbf{h}_{N_{+}-1}\left(r\right)}{\left(\boldsymbol{\rho}_{+}
\left(s\right),1\right)\mathbf{h}_{N_{+}}\left(r\right)}\right).
\]
 Hence,
\[
D'_{x}\left(s,u\right)=\frac{\rho_{N_{+}}^{+}\left(s\right)}
{\rho_{N_{+}}^{+}\left(s+u\right)}\left(\boldsymbol{\rho}_{+}
\left(s+u\right)-\boldsymbol{\rho}_{+}\left(s\right)\right)
e^{\mathbf{R}_{+}\left(s\right)x}\mathbf{e},\;
x>0,
\]
and taking into account that $\lim_{x\rightarrow+\infty}
D_{x}\left(s,u\right)=1$
we deduce (\ref{eq:3.2}).
\end{proof}
This statement generalize the representation of $D_{x}\left(s,u\right)$ known
for almost semi-continuous processes given in~\cite{Husak2011engl}.

Note that, for a non step-wise L\'evy process $\mathsf{P}\left\{ X_{t}=0\right\} =0$,
then by \cite[VI, Lemma~15]{Bertoin1996} for any $t\geq0$ the time
it spends in $\left[0,\infty\right)$ $Q_{0}\left(t\right)=
\int_{0}^{t}I\left\{ X_{v}\geq0\right\} dv$
and the instant of its last supremum $g_{t}=
\sup\left\{ v<t:X_{v}=X_{v}^{+}\right\} $
have the same law. Moreover, by \cite[Theorem 2.9]{Husak2011engl} $Q_{0}\left(t\right)$ and the time the maximum is achieved
$T_{t}=\inf\left\{ v>0:X_{v}=X_{v}^{+}\right\} $ also have the same
law. Hence, the results of Theorem \ref{thm:3.1} define the moment
generating functions of $g_{\theta_{s}}$ and $T_{\theta_{s}}$.

Let $L\left(t\right)$ be the local time in $\left[0,t\right]$ that
$X_{t}^{+}-X_{t}$ spends at zero and $$L^{-1}\left(t\right)=
\inf\left\{ v>0:L\left(v\right)>t\right\} $$
is the inverse local time (for details, see \cite[VI]{Bertoin1996}).
Denote by $\kappa\left(s,r\right)$ the Laplace exponent of the so
 called
ladder process $\left\{ L^{-1},X_{L^{-1}}\right\}$:
 $$e^{-\kappa\left(s,r\right)}=\mathsf{E}\left[e^{-sL^{-1}
\left(1\right)-rX_{L^{-1}\left(1\right)}},1<L_{\infty}\right].$$
According to \cite[VI, (1)]{Bertoin1996}:
\[
\mathsf{E}e^{-rX_{\theta_{s}}^{+}-ug_{\theta_{s}}}=
\frac{\kappa\left(s,0\right)}{\kappa\left(s+u,r\right)}.
\]
Assuming that the normalization constant of the local time is 1, we
can deduce that $\kappa\left(s,0\right)=
\mathsf{E}e^{-\left(1-s\right)g_{\theta_{s}}}$.
Taking into account that for non step-wise processes
$Q_{0}\left(\theta_{s}\right)$
and $g_{\theta_{s}}$ have the same distribution we can
write that
\begin{equation}\label{eq:3.4}
\kappa\left(s,r\right)=\frac{D_{0}\left(s,1-s\right)}
{\mathsf{E}e^{-rX_{\theta_{s}}^{+}}}.
\end{equation}
Hence, using Theorem \ref{thm:2.1} and Theorem \ref{thm:3.1}, we can deduce
the following statement.
\begin{cor}\label{cor:3.1}
If one of the cases $\left(NS\right)_{-}$
 or $\left(S\right)_{-}$
$\left(a>0\right)$ holds, then
\begin{multline*}
\kappa\left(s,-r\right)=\frac{\rho_{N_{-}}^{-}\left(1\right)}
{\rho_{N_{-}}^{-}\left(s\right)}\times\\\times\left(s-r\left(A_{*}^{-}\left(s\right)
+\mathsf{E}e^{rX_{\theta_{s}}^{-}}\tilde{\Pi}^{+}\left(r\right)
-\mathbf{q}_{-}\left(s\right)\left(r\mathbf{I}+
\mathbf{R}_{-}\left(s\right)\right)^{-1}
\tilde{\Pi}^{+}\left(-\mathbf{R}_{-}\left(s\right)\right)
\mathbf{e}\right)\right).
\end{multline*}

If $X_{t}$ is a compound Poisson process with negative drift $a<0$,
 without gaussian part $\left(\sigma=0\right)$, and with bilateral
ME distributed jumps, then
\[
\kappa\left(s,r\right)=\frac{\rho_{d_{-}+1}^{-}\left(1\right)}
{\rho_{d_{-}+1}^{-}\left(s\right)}\frac{\beta_{d_{+}}^{+}}
{\rho_{d_{+}}^{+}\left(s\right)}\left(1+\frac{\left(\boldsymbol{\rho}_{+}
\left(s\right)-\boldsymbol{\beta}_{+}\right)\mathbf{h}_{d_{+}-1}
\left(r\right)}{\left(\boldsymbol{\rho}_{+}\left(s\right),1\right)
\mathbf{h}_{d_{+}}\left(r\right)}\right).
\]
\end{cor}
The next statement applies Theorem \ref{thm:3.1} and
 Corollary~\ref{cor:2.3}
to get a representation of the moment generating function of the total
sojourn time over a level $D_{x}\left(0,u\right)=
\mathsf{E}e^{-u\int_{0}^{\infty}I\left\{ X_{t}>x\right\} dt}$,
which in risk theory defines the time in risk zone
(for details, see
\cite{Husak2011engl}).
\begin{cor}\label{cor:3.2}
If one of the cases $\left(NS\right)_{-}$ or $\left(S\right)_{-}$
$\left(a>0\right)$ holds and $\mu=\mathsf{E}X_{1}<0$, then
\[
\mathsf{E}e^{-u\int_{0}^{\infty}I\left\{ X_{t}>0\right\} dt}
=\frac{|\mu|}{u}\frac{\prod_{i=1}^{N_{-}}r_{i}^{-}
\left(u\right)}{\prod_{i=2}^{N_{-}}r_{i}^{-}},
\]
\begin{equation}\label{eq:3.5}
\mathsf{E}e^{-u\int_{0}^{\infty}I\left\{ X_{t}>x\right\} dt}
=\frac{|\mu|\prod_{i=1}^{N_{-}}r_{i}^{-}\left(u\right)}
{u\prod_{i=2}^{N_{-}}r_{i}^{-}}\left(\boldsymbol{\rho}_{-}
\left(u\right)-\boldsymbol{\rho}_{-}\left(0\right)\right)
\mathbf{R}_{-}^{-1}\left(u\right)e^{\mathbf{R}_{-}\left(u\right)x}
\mathbf{e},x<0.
\end{equation}
The integral transform of the moment generating function of the sojourn
time over a positive level has the next representation
\begin{multline}\label{eq:3.6}
\int_{-0}^{+\infty}e^{rx}D'_{x}\left(0,u\right)dx+
D_{0}(0,u)=\frac{\mathsf{E}e^{rX^{+}}}
{\mathsf{E}e^{rX_{\theta_{u}}^{+}}}=\\
=\frac{1-\frac{r}{u}\left(A_{*}^{-}\left(u\right)
+\mathsf{E}e^{rX_{\theta_{u}}^{-}}\tilde{\Pi}^{+}
\left(r\right)-\mathbf{q}_{-}\left(u\right)
\left(r\mathbf{I}+\mathbf{R}_{-}\left(u\right)\right)^{-1}
\tilde{\Pi}^{+}\left(-\mathbf{R}_{-}\left(u\right)\right)
\mathbf{e}\right)}{1-r\left(A'_{-}+p'_{-}
\tilde{\Pi}^{+}\left(r\right)+\mathbf{q}'_{-}\left(r\mathbf{I}+
\mathbf{R}_{-}\left(0\right)\right)^{-1}
\left(\tilde{\Pi}^{+}\left(r\right)\mathbf{I}-\tilde{\Pi}^{+}
\left(-\mathbf{R}_{-}\left(0\right)\right)\right)\mathbf{e}\right)}.
\end{multline}
If for the process $X_t$: $\sigma=0,a\leq0$, negative jumps
have a ME distribution and $\mu<0$, then
\begin{multline}\label{eq:3.7}
\mathsf{E}e^{-u\int_{0}^{\infty}I\left\{ X_{t}>x\right\} dt}=
\mathsf{P}\left\{ X^{+}<x\right\} +
\frac{p_{+}\left(u\right)}{u}\bigg(p_{-}\left(u\right)
\int_{0}^{x}\overline{\Pi}^{+}\left(x-z\right)
\mathsf{P}\left\{ X^{+}\in dz\right\} +\\
+\mathbf{q}_{-}\left(u\right)\int_{0}^{\infty}
\overline{\Pi}^{+}\left(y\right)\int_{\max\left\{ 0,x-y\right\} }^{x}
e^{\mathbf{R}_{-}\left(u\right)(x-y-z)}\mathsf{P}\left\{ X^{+}\in dz\right\}
 \mathbf{e} dy\bigg).
\end{multline}
\end{cor}
\begin{proof} Equality (\ref{eq:3.5}) follows by taking the limit as $s\rightarrow 0$
in (\ref{eq:3.1}). Formula (\ref{eq:3.6}) is a straightforward consequence
of formulas (\ref{eq:2.4}), (\ref{eq:2.16}) and (\ref{eq:3.3}).

If $\sigma=0,a\leq0$, then $\left\{ \tau_{0}^{+},X_{\tau_{0}^{+}}\right\} $
has non degenerate joint distribution. Applying prelimit generalization
of the Pollaczek-Khinchin formula (\cite[Theorem 2.4]{Husak2011engl}),
we get
\[
\int_{-0}^{+\infty}e^{rx}D'_{x}\left(s,u\right)dx+
D_{+0}(s,u)=\frac{\mathsf{E}e^{rX_{\theta_{s}}^{+}}}
{\mathsf{P}\left\{ X_{\theta_{s+u}}^{+}=0\right\} }
\left(1-\mathsf{E}\left[e^{-\left(s+u\right)\tau_{0}^{+}
+rX_{\tau_{0}^{+}}},\tau_{0}^{+}<\infty\right]\right).
\]
Whence $D_{+0}(s,u)=\frac{\mathsf{P}\left\{ X_{\theta_{s}}^{+}
=0\right\} }{\mathsf{P}\left\{ X_{\theta_{s+u}}^{+}=0\right\}}$
and for $x>0$
\[
D_{x}\left(s,u\right)=
P_{+}\left(s,x\right)+\int_{0}^{x}
\mathsf{P}\left\{ X_{\theta_{s+u}}^{+}>0,
X_{\tau_{0}^{+}}>x-z\right\} dP_{+}\left(s,z\right).
\]
Due to \cite[Corollary 2.3]{Husak2011engl}:
\[
\mathsf{P}\left\{ X_{\theta_{s+u}}^{+}>0,X_{\tau_{0}^{+}}>z\right\}
=\frac{\mathsf{P}\left\{ X_{\theta_{s+u}}^{+}=0\right\} }
{s+u}\int_{-\infty}^{0}
\overline{\Pi}^{+}\left(z-y\right)dP_{-}\left(s+u,y\right).
\]
 If negative jumps have the ME distribution, then
\begin{multline*}
D_{x}\left(s,u\right)
=P_{+}\left(s,x\right)+\frac{p_{+}\left(s+u\right)}{s+u}
\bigg(p_{-}\left(s+u\right)\int_{0}^{x}\overline{\Pi}^{+}
\left(x-z\right)dP_{+}\left(s,z\right)+\\
+\mathbf{q}_{-}\left(s+u\right)\int_{0}^{\infty}\overline{\Pi}^{+}
\left(y\right)\int_{\max\left\{ 0,x-y\right\} }^{x}
e^{\mathbf{R}_{-}\left(s+u\right)(x-y-z)}\mathbf{e}dP_{+}\left(s,z\right)dy\bigg).
\end{multline*}
from here as $s\rightarrow0$ relation (\ref{eq:3.7}) follows.
\end{proof}

Note that, for the step-wise $\left(a=0\right)$ almost
lower semi-continuous
processes formula (\ref{eq:3.7}) is reduced to
\begin{multline*}
\mathsf{E}e^{-u\int_{0}^{\infty}I\left\{ X_{t}>x\right\} dt}=
\mathsf{P}\left\{ X^{+}<x\right\} +\frac{1}{u+\lambda}
\bigg(\int_{0}^{x}\overline{\Pi}^{+}\left(x-z\right)
\mathsf{P}\left\{ X^{+}\in dz\right\} +\\
+\left(b_{1}-r_{1}^{-}\left(u\right)\right)
\int_{0}^{\infty}\overline{\Pi}^{+}\left(y\right)
\int_{\max\left\{ 0,x-y\right\} }^{x}e^{r_{1}^{-}\left(u\right)(x-y-z)}
\mathsf{P}\left\{ X^{+}\in dz\right\} dy\bigg).
\end{multline*}

If negative (positive) jumps have hyperexponential distribution, that
is, if we have additional condition that $b_{m_{-}}>\ldots>b_{2}>b_{1}>0$
($c_{m_{+}}>\ldots>c_{2}>c_{1}>0$), then the roots of the cumulant
equation $\left\{ -r_{i}^{-}\left(s\right)\right\} _{i=1}^{N_{-}}$
$\left(\left\{ r_{i}^{+}\left(s\right)\right\} _{i=1}^{N_{+}}\right)$
are real and distinct (see \cite{Lewis2008}), and the matrix exponents in Theorem~\ref{thm:3.1}
can be simplified.
\begin{cor}
\label{cor:3.3} If one of the cases $\left(NS\right)_{-}$ or
$\left(S\right)_{-}$
$\left(a>0\right)$ holds and $b_{m_{-}}>\ldots>b_{1}>0$, then
\[
D_{0}(s,u)=\frac{s}{s+u}\prod_{i=1}^{N_{-}}
\frac{r_{i}^{-}\left(s+u\right)}{r_{i}^{-}\left(s\right)},
\]
\begin{equation}\label{eq:3.8}
D_{x}\left(s,u\right)=\frac{s}{s+u}\left(1+\sum_{k=1}^{N_{-}}
\frac{\prod_{i=1}^{N_{-}}\left(r_{k}^{-}\left(s+u\right)/r_{i}^{-}
\left(s\right)-1\right)}{\prod_{i=1,i\neq k}^{N_{-}}
\left(r_{k}^{-}\left(s+u\right)/r_{i}^{-}\left(s+u\right)-1\right)}
e^{r_{k}^{-}\left(s+u\right)x}\right),x<0.
\end{equation}
For $\mu<0$
\[
D_{0}(0,u)=\frac{|\mu|}{u}
\prod_{i=2}^{N_{-}}\frac{r_{i}^{-}\left(u\right)}{r_{i}^{-}},
\]
\begin{equation}\label{eq:3.9}
D_{x}\left(0,u\right)=\frac{|\mu|}{u}\sum_{k=1}^{N_{-}}
\frac{\prod_{i=2}^{N_{-}}\left(r_{k}^{-}\left(u\right)/r_{i}^{-}-1\right)}
{\prod_{i=1,i\neq k}^{N_{-}}\left(r_{k}^{-}\left(u\right)/
r_{i}^{-}\left(u\right)-1\right)}
r_{k}^{-}\left(u\right)e^{r_{k}^{-}\left(u\right)x},x<0.
\end{equation}

If one of the cases $\left(NS\right)_{+}$ or $\left(S\right)_{+}$
$\left(a<0\right)$ holds and $c_{m_{+}}>\ldots>c_{1}>0$, then
\[
D_{0}(s,u)=\prod_{i=1}^{N_{+}}\frac{r_{i}^{+}\left(s\right)}
{r_{i}^{+}\left(s+u\right)},
\]
\begin{equation}\label{eq:3.10}
D_{x}\left(s,u\right)=1-\sum_{k=1}^{N_{+}}
\frac{\prod_{i=1}^{N_{+}}\left(1-r_{k}^{+}\left(s\right)/r_{i}^{+}
\left(s+u\right)\right)}{\prod_{i=1,i\neq k}^{N_{+}}
\left(1-r_{k}^{+}\left(s\right)/r_{i}^{+}
\left(s\right)\right)}e^{-r_{k}^{+}\left(s\right)x}\:,x>0.
\end{equation}
For $\mu>0$ and $x\geq0$: $D_{x}\left(0,u\right)=0,$ in other words
$\mathsf{P}\left\{ \int_{0}^{+\infty}I\left\{ X_{t}>x\right\}
 dt=+\infty\right\} =1$.
\end{cor}
\begin{proof} If we have one of the cases $\left(NS\right)_{-}$
or $\left(S\right)_{-}$ $\left(a>0\right)$ and
$b_{m_{-}}>\ldots>b_{2}>b_{1}>0$, then
the roots $\left\{ -r_{i}^{-}\left(s\right)\right\} _{i=1}^{N_{-}}$
are real and distinct, and instead of using formula (\ref{eq:3.1})
it is more convenient to substitute relation (\ref{eq:2.6}) in (\ref{eq:3.3}) and invert with
respect to $r$. Similarly for the case $\left(NS\right)_{+}$ or $\left(S\right)_{+}$
$\left(a<0\right)$ and $c_{m_{+}}>\ldots>c_{2}>c_{1}>0$ we can deduce
formula (\ref{eq:3.10}). To get $D_{x}\left(0,u\right)$ in the corresponding
cases apply the limit behavior of the roots of cumulant equation
as $s\rightarrow0$. To find the limit as $s\rightarrow0$ in (\ref{eq:3.10}) for the case
$\left(NS\right)_{+}$ or $\left(S\right)_{+}$ and for $\mu>0$ we can use the relation
\[
\frac{\prod_{i=1}^{N_{+}}\left(1-r_{k}^{+}\left(s\right)/r_{i}^{+}
\left(s+u\right)\right)}{\prod_{i=1,i\neq k}^{N_{+}}
\left(1-r_{k}^{+}\left(s\right)/r_{i}^{+}
\left(s\right)\right)}\underset{s\rightarrow0}{\longrightarrow}
\begin{cases}
1 & k=1,\\
0 & k\neq1.
\end{cases}
\]
\end{proof}

Note that, using the results of \cite{Kuznetsov2012},
 Corollary \ref{cor:3.3}
could be generalized for the case of the so called meromorphic L\'evy
processes (the cumulant function is holomorphic except a
set of isolated points, the poles of the function),
 for which $N_{\pm}=\infty$
in (\ref{eq:3.8})--(\ref{eq:3.10}).


\begin{thebibliography}{10}
\bibitem{Kyprianou2006}
A. E. Kyprianou,
\textit{Introductory Lectures on Fluctuations of Levy Processes with Applications},
Springer,
New York,
{2006}.

\bibitem{Bratiychuk1990}
N. Bratiychuk and D. Husak,
\textit{Boundary-Values Problems for Processes with Independent Increments},
 Naukova Dumka,
 Kyiv,
1990 (in Russian).

\bibitem{Lewis2008}
A. L. Lewis and E. Mordecki,
\textit{Wiener-Hopf factorization for Levy processes having negative jumps with rational transforms},
J. of Appl. Prob.
\textbf{45}
(2008),
no.~1,
118--134.

\bibitem{Landriault2011}D. Landriault, J.-F.Renaud and X. Zhou,
\textit{Occupation times of spectrally negative L\'evy
processes with applications},
 Stochastic processes and their applications
\textbf{212(11)}
(2011),
 2629--2641.

\bibitem{Husak2011engl}
D. Husak, \textit{Processes with Independent Increments in Risk Theory},
Institute of Mathematics of the NAS of Ukraine, Kyiv,
2011 (in Ukrainian).

\bibitem{Kuznetsov2012}
A. Kuznetsov, A. E. Kyprianou, J. C. Pardo,
\textit{Meromorphic Levy processes and their fluctuation identities},
Ann. of Appl. Prob.
\textbf{22}
(2012),
no.~3,
1101--1135.

\bibitem{Asmussen2010}
S. Asmussen and  H. Albrecher,
\textit{Ruin Probabilities},
World Scientific,
Singapore,
2010.

\bibitem{Fackrell2003}
M. W. Fackrell, \textit{Characterization of Matrix-exponential Distributions},
PhD Thesis, The University of Adelaide, 2003.



\bibitem{Bladt2012}M. Bladt and B.F. Nielsen,
\textit{Multivariate matrix–-exponential distributions},
Stochastic Models,
\textbf{26}
(2010),
no.~1,
1-–26.

\bibitem{Bertoin1996}
J. Bertoin,
\textit{L\'evy Processes},
Cambridge Univ. Press,
Cambridge,
1996.


\end{thebibliography}
\end{document}